\documentclass[12pt]{article}
\usepackage{amsmath, amssymb, amsfonts}
\usepackage[notref, notcite]{showkeys}

\usepackage{xcolor}
\usepackage{geometry}                		
\usepackage{graphicx}				
\usepackage{amssymb}
 at 10truept
\font \symb = msbm10 at 12truept

 at 16truept
 at 16truept
 at 12truept

 at 10truept
 at 10truept
 at 8truept
 at 10truept
 at 8truept
 at 10 truept
 at 11 truept
 at 9truept
 at 12truept
 at 12truept
 at 7truept
 at 14 truept
 at 12truept
 at 9truept
 at 14truept
 at 17 truept

\newcommand\R{\hbox {\symb R}}

\def\Xint#1{\mathchoice
{\XXint\displaystyle\textstyle{#1}}{\XXint\textstyle\scriptstyle{#1}}
{\XXint\scriptstyle\scriptscriptstyle{#1}}{\XXint\scriptscriptstyle\scriptscriptstyle{#1}}\!\int}
\def\XXint#1#2#3{{\setbox0=\hbox{$#1{#2#3}{\int}$ }
\vcenter{\hbox{$#2#3$ }}\kern-.6\wd0}}

\def\abstr{\if@twocolumn
\section*{abstr}
\else \small
\quotation\fi
{\bf Abstract.}
}
\def\endabstr{\if@twocolumn\else\endquotation\fi}
\mark{{}{}}
\begin{document}
	\title{{\LARGE {On Harnack inequality  to the homogeneous nonlinear degenerate parabolic equations}}\\
		[30pt]}
	\author{Jasarat Gasimov$^{a}$, Farman Mamedov$^{b,c}$ 
		\\
		{\small {
			jasarat.gasimov@emu.edu.tr \quad	mfarmannn@gmail.com  
		}}\\
		[1cm]$^{a}$Department of Mathematics, Eastern Mediterranean University,\\
		(Mersin 10, 99628, T.R. North Cyprus, Turkey)
		\\
		$^{b}$Mathematics and Mechanics Institute  \\ 
		(\emph{\ B.Vahabzade str., 9, AZ 1141, Baku, Azerbaijan}
		)\\ \&    $^{c}$State Oil and Industry University \\
		(\emph {\  Azadlig ave. 34, Baku AZ1010, Azerbaijan}) 
	} 
	\maketitle
	\numberwithin{equation}{section}
	\newtheorem{teo}{\quad Theorem}\newtheorem{prop}[teo]{Proposition} %
	\newtheorem{defi}[teo]{Definition} \newtheorem{lem}[teo]{Lemma} %
	\newtheorem{cor}[teo]{Corollary} \newtheorem{rem}[teo]{Remark} %
	\newtheorem{ex}[teo]{Example} \newtheorem{cla}{Claim}
	\newtheorem{con}[teo]{Conjecture}
	\newtheorem{prob}[teo]{Problem}
	
	\pagestyle{myheadings} \baselineskip 13.5pt
	
	\bigskip \bigskip
	
	\thispagestyle{empty} \medskip \indent
	
	\begin{abstr} \small{In this paper, the Harnack inequality result is established for a new class of the homogeneous nonlinear degenerate parabolic equations
			$$
		\text{div} \, A(t,x,u,\nabla_x u)-\partial_t \,  \vert u\vert^{p-2}u=0 
			$$
			on a bounded domain $ D \subset \R^{n+1}. $ Let $A(t,x,\xi,\eta)$ be measurable function on $\mathbb{R}\times\mathbb{R}^n\times \mathbb{R}\times\mathbb{R}^n\to\mathbb{R}^n$ that satisfies the Caratheodory conditions
			for $ \, \text{arbitrary } \,  (t,x)\in D$ and $(\xi,\eta)\in\mathbb{R}^{1}\times\mathbb{R}^n.$ The following growth conditions are also satisfied:
			\begin{equation*}
			A(t,x,\xi,\eta)\eta\geq c_{1}\omega(t,x)\vert\eta\vert^{p}
			\end{equation*}
		\begin{equation*}
			\vert A(t,x,\xi,\eta)\vert\leq c_{2}\omega(t,x)\vert\eta\vert^{p-1},\quad p>1.
		\end{equation*}
			 The exclusive Muckenhoupt condition $ \omega^{\alpha}  \in  A_{1+{\alpha}/r} : $
			$$ 
			\left (\Xint -_{Q_T} \omega^{\alpha} dxdt \right )^{1/{\alpha}}\left (\Xint -_{Q_T} \sigma^r dxdt \right )^{(p-1)/r} \leq c_0
			$$  is assumed for a $ {\alpha} > (n+p)/p, $  an $ r>n(p-1)/p $ such that $ n(p-1)/pr+(n+p)/p{\alpha}<1; $ the cylinders  
			$$ \left \{ Q_T= K_R^{x_0} \times (t_0-T, t_0):  (t_0, x_0) \in D , \, T=CR^p\big/ \left  ( \Xint -_{Q_T}\omega^{\alpha} \, dxdt \right )^{1/{{\alpha}}} \right \},   $$  $  \sigma=\omega^{-\frac{p^{\prime}}{p}}  ,  \, K_r^{x}=\{ y\in \R^n: \, |y-x|<r \} .  $ The cylinders are of $T \to 0 $ as $R\to 0 .$ 
		}
	\end{abstr}
	
	\bigskip
	
	\bigskip
	
	\bigskip
	
	
	\textbf{MS Subject classification:} {35K65
		35B65,  35D10, 35K10 }
	
	\vspace{2cm} \quad
	
	\bigskip
	
	\noindent
	
\tableofcontents
	
	\section{Introduction.}
	
	\label{Sec1}
	
	\bigskip
	
	\qquad
	In this paper, we study the Harnack inequality estimate for the weak solutions of the  homogeneous nonlinear degenerate parabolic equation 
	\begin{equation} \label{PE}
 \text{div}\,  A(t,x,u,\nabla_x u)-\partial_t \, \vert u\vert^{p-2}u=0 
	\end{equation}
	on a bounded open domain $ D \subset \R^{n+1}, \,  n\geq 1, $ and in case of $p=2,$ we quote to Mamedov \cite{FM} in the sense that, in fact, one condition \eqref{e4} in $p=2$ only he had used to characterize the Harnack inequality.   We suppose such approach would become fruitful
on handing the parabolic analogue of the well known
Fabes-Kenig-Serapioni’s result in elliptic equations\cite{FKS}. While during the 1980s in the series of works by Chiarenza and Serapion  \cite{ChiS,ChiSer2,ChS,ChS2}  two separate conditions were applied for $\omega $ to characterize the Harnack inequality. Those conditions looks like to be exact for case of factorised $\omega=\omega_1(x)  \omega_2(t).$ However, as it was shown in \cite {FM} such conditions become no exact for non-factorised $\omega=(\vert x \vert^2  + t^{2})^{1/2} .$ Precisely, the cited author provides an example $\omega$ which satisfies \eqref{e4} and violates  the conditions of Chiarenza-Serapioni.

	Because of nonlinear nature of the term $\partial_t \vert u\vert^{p-2}u$, we are unable to simply add a constant to a solution. To the best of our  knowledge , Trudinger was the first to investigate this type of equation in \cite{nst}, where he established a Harnack inequality for weak solutions that are nonnegative. The demonstration relied on Moser's renowned contribution \cite{Mo} and employed a parabolic adaptation of the John-Nirenberg lemma. Two decades later, Fabes and Garofalo presented a simplified proof of the parabolic John-Nirenberg lemma, documented in \cite{FKS1}.This issue has been extensively explored in the literature. In the context of homogeneous nonlinear equations, studies have been conducted by \cite{tjj,kkuusi,tkrlsu,nst,vvespri}, and for the nonhomogeneous cases, the equation was studies in \cite{DGV,MS}, while relevant results for elliptic equations have been obtained in \cite{FKS,AM1,AM2,DIF2,DIF,DIF3}.

	We should mention that the equation is homogeneous, which implies that we obtain the Harnack inequality without the need for "intrinsic".
	
$\odot$	Initially, we establish the Moser inequality through the utilization of weighted interpolation inequality and the iteration lemma to accomplish our primary objective.

$\odot$ Next, we demonstrate the Harnack  inequality based on the foundation provided by the  Bombieri lemma.
	
	$\bullet$Let $A(t,x,\xi,\eta):\mathbb{R}\times\mathbb{R}^n\times \mathbb{R}\times\mathbb{R}^n\to\mathbb{R}^n$ be Caratheodory function satisfying the following growth conditions:
	
	There are positive constants $c_{1},c_{2}$ such that 
	\begin{equation}\label{grw1}
		A(t,x,\xi,\eta)\eta\geq c_{1}\omega(t,x)\vert\eta\vert^{p}
	\end{equation}
	\begin{equation}\label{grw2}
		\vert A(t,x,\xi,\eta)\vert\leq c_{2}\omega(t,x)\vert\eta\vert^{p-1},\quad p>1.
	\end{equation}
for $ \, \text{arbitrary } \,  (t,x)\in D$ and $(\xi,\eta)\in\mathbb{R}^{1}\times\mathbb{R}^n.$ 

$\bullet$There exist a  \, $ c_0>0 $ a $ \alpha > (n+p)/p,  \, r > n(p-1)/p $ with $n(p-1)/pr+(n+p)/p\alpha <1$ such that the condition 
 
		\begin{equation}\label{e4}
			\omega^{\alpha} \in A_{1+{\alpha}/r}: \quad \left  (\Xint -_{Q_T}\omega^{\alpha} dxdt \right )^{1/{\alpha}}\left (\Xint -_{Q_T}\sigma^r dxdt \right )^{(p-1)/r}\leq c_0 
	\end{equation}  is fulfilled all over the cylinders 
	\begin{equation}\label{cyl}
		Q_T=K^{x}_R\times (t-T, t), \quad  T=CR^p\Big/ \left (\Xint -_{Q_T}\omega^{\alpha} \, dxdt \right )^{1/{\alpha}}; 
	\end{equation}and $\sigma=\omega^{-\frac{p^{\prime}}{p}} $with center $(t, x)$
	
	\smallskip
	
	The value of $ T=T(R) $ depending on $R$ can be determined as a solution of the equation  
	\begin{equation}\label{Tdef}
		\left (\iint\limits _{Q_T} \omega^{\alpha} \, dtdx\right )^{1/{\alpha}} =  \frac{ C R^{n/{\alpha}+p}}{T^{1/{\alpha}^\prime}} ,
	\end{equation}
 which is unique as center $(t, x)$ and $R$ be fixed; $C>0$ is a fixed constant.  This follows from the increasing of left hand side of \eqref{Tdef} as a function of $T$ and vanishing near origin, while the right hand side is decreasing and tends to infinite as $T \to 0 . $
	Also the $T=T(R)$ may be defined as   
	\begin{equation}\label{dT}
		T=\sup \left \{ F>0: \,   \left (\iint\limits_{Q_F} \omega^{\alpha} \, dtdx \right )^{1/{\alpha}} F^{1/{\alpha}^\prime}\leq \, C R^{n/{\alpha}+p}   \right \} .
	\end{equation}

In this paper, we demonstrate that local weak solutions, which are non-negative, to equation \eqref{PE}, adhere to the Harnack inequality within the corresponding cylinders, as defined in \eqref{cyl}. The size of these cylinders varies depending on the weight parameter $\omega$. 

We adopt standard notation. Let $p'$ denote the number such that $\frac{1}{p} + \frac{1}{p'} = 1$ if $1 < p < \infty$, and $p' = \infty$ if $p = 1$. We define the function $v: \mathbb{R}^n \rightarrow (0,\infty)$ to satisfy the $A_\infty$-Muckenhoupt condition if there exist constants $C, \delta > 0$ such that for any ball $K \subset \mathbb{R}^n$ and measurable subset $E \subset K$, the inequality 
\begin{align*}
	v(E) \leq C \left( \frac{\vert E \vert}{\vert K \vert} \right)^\delta v(K)
\end{align*}
holds.

Here and henceforth, for a positive function $v$, the notation $v(E)$ signifies $\int_{E}v\,dx$, and for a measurable set $E$, $\vert E\vert$ represents the Lebesgue measure of set $E$. For a measurable set $E$ and an integrable function $f$, we denote $\Xint -_{E}f\,dx=\frac{1}{\vert E\vert}\int_{E}f\,dx$. The notation $K(y,r)$ (or $K_{r}^{y}$) refers to the $n$-dimensional Euclidean ball of radius $r$ centered at $y$, defined as $K(y,r)=\{x\in \mathbb{R}^{n}:\vert x-y\vert<r\}$.

For the cylinder $Q_{T}=K_{R}^{x_{0}}\times (t_{0}-T,t_{0})\subset \mathbb{R}^{n+1}$, we denote $\Gamma(Q_{T})$ as its parabolic boundary, given by $\partial K_{R}^{x_{0}}\times(t_{0}-T,t_{0})\cup K_{R}^{x_{0}}\times\{t_{0}-T\}$. For a domain $D\subset \mathbb{R}^{n}$ (or $\Omega\subset \mathbb{R}^{n+1}$), we define $L^{p}(D)$ (or $L^{p}(\Omega)$) as the space of measurable functions on $D$ (or $\Omega$) with finite norms, expressed as $(\int_{D}\vert f\vert^{p}\,dx)^{{1}/{p}}$ (or $(\int_{\Omega}\vert f\vert^{p}\,dtdx)^{{1}/{p}})$. Furthermore, we denote $W^{1}_{p,\omega}(D)$ as the weighted Sobolev space on $D$, equipped with the finite norm $\|f\|_{L^{p}(D)} + \|\omega^{1/p} |\nabla f|\|_{L^{p}(D)}$.

 We would like to note that we will pursue the methods and ideas which was indicaded in \cite{FM}.
	\begin{defi}
		We say
		\begin{align*}
			u\in L^{\infty}\bigg(t_{0}-T,t_{0};L^{p}\bigg(K_{R}^{x_{0}}\bigg)\bigg)\cap L^{p}\bigg(t_{0}-T,t_{0};W^{1}_{p,\omega}\bigg(K_{R}^{x_{0}}\bigg)\bigg)
		\end{align*}
	is a sub-solution (super-solution) of (\ref{PE}) on $Q_{T}$ if $
	\forall v>0$ from the space $Lip(Q_{T})$ vanishing on $\Gamma(Q_{T})$ satisfies:
	\begin{align}\label{dfn}
	-\int_{K_R^{x_0}}u^{p-1}v\Big\vert_{t=t_0 -T} ^{t=t_0}-\iint_{Q_T}A(t,x,u,\nabla_{x}u)\nabla_{x}vdtdx+\iint_{Q_T}u^{p-1}\partial_{t}vdtdx\geq(\leq)0.
	\end{align}
	\end{defi}
We define a weak solution of equation (\ref{PE}) as a function $u$ that satisfies both the sub-solution and super-solution conditions simultaneously.

Throughout the paper, the symbol C will represent a constant that depends on various parameters, such as the dimension n, exponents p and r, and constants $c_{1}$, $c_{2}$, and $c_{0}$. It is important to note that this constant may vary at different instances within the paper, and its specific value is determined accordingly.

	\bigskip
	
	\noindent
	
	\section{Preliminary}
	
	\label{Sec2}
	
	\bigskip

	\qquad
Interpolation inequalities, particularly those formulated on cylindrical domains in $\mathbb{R}^{n+1}$, constitute fundamental tools in the analysis of parabolic equations. These inequalities have proven indispensable in exploring the regularity characteristics of degenerate parabolic equations.
	
	Let $D$ be an open domain in $\mathbb{R}^n$ (which may or may not be bounded). Consider three positive measurable functions, $v$, $\omega_1$, and $\omega_2$. These functions satisfy the following properties:
	
	$\bullet$The functions $v$ and $\omega_2$ belong to the Muckenhoupt $A_\infty$ class of functions defined on $\mathbb{R}^n$.
	
	$\bullet$The function $\sigma_1=\omega_1^{1-s'}$ is locally integrable, meaning $\sigma_1 \in L^{1, \text{loc}}$, where $s'$ denotes the conjugate exponent of the Sobolev space $L^s$.
	
	Based on these assumptions, we propose an inequality involving the weighted interpolation of these functions.
	
 Denote the balls system of $\R^n$  $$\mathcal{F} _D=\{K=K(x, r):  \, x\in D, \, 0<r<\text{diam} \, D \} $$ in order to propose  
	\begin{teo}\cite{FM} \label{t1}
		Let $m>0, \, s \geq 1, \, q\geq 1+\frac{m}{s^\prime}. $ There exists a constant $C_0$ depending only on $n, s, q, m $ and $A_\infty $-constants of the functions 
		$ v, \omega_2 $ such that if  
		\begin{equation}\label{ee3}
			\vert K\vert^{1/n-1}v(K)\sigma_1(K)^{1/s^\prime} \leq A \, \omega_2(K)^{\frac{q-1}{m}}, \quad  \forall K\in \mathcal{F}_D,
		\end{equation}
		then
		\begin{equation}\label{int}
			\big\Vert f \big \Vert_{q, v}\, \leq \, C_0 \, A^{1/q}\,  \, \, \big\Vert f \big \Vert_{m, \omega_2}^{\frac{1}{q^\prime}} \, \big\Vert  \nabla_{x}f \big \Vert_{s, \omega_1}^{\frac{1}{q}}, \quad f\in Lip_0(D).
		\end{equation}
	\end{teo}

		\begin{cor}\cite{FM} \label{cc40}
			Let $ q\in \left [ 1, (n+p)/n \right ]. $ Then for a function $f\in Lip_0(D)$ the 
			the inequality
			\begin{equation}\label{int4}
				\big\Vert f \big \Vert_{q}\, \leq \, C_0 \, A^{1/q} \big\Vert f \big \Vert_{p}^{\frac{1}{q^\prime}}\, \, 
				\,  \big\Vert  \nabla_{x} f \big \Vert_{1}^{\frac{1}{q}}
			\end{equation} holds, 
			where\,$  A=\vert D \vert ^{\frac{1}{n}-\frac{q-1}{p}}, \, \, 
			C_0$ depends only on $n, q. $
		\end{cor}

		\begin{lem} \cite{FM}  \label{lem1}Let $q\in \left [ 1, (n+p)/n \right ] $ and let $D\subset \R^n$ be an open domain $ Q_T=D\times (t_0-T, t_0) $ be the cylinder with $T$ from \eqref{Tdef},  The positive functions $\omega, \sigma=\omega^{-\frac{p^{\prime}}{p}} \in L^{1, loc}(\R^{n+1}).$ Then for a function $ u\in Lip (Q_T) $ vanishing on $\Gamma (Q_T) $ the inequality 
			\begin{equation}\label{equ1}
				\left  \Vert  u  \right \Vert _{L^{q}(Q_T)}^{q}\leq C \left ( \sup_{t\in (t_0-T, t_0)}\,  \left \Vert u(t, \cdot ) \right \Vert _{L^{p}\left (D \right )}^{q-1} \right ) \, \, \left \Vert  \omega ^{1/p} \vert \partial _x u \vert \right \Vert _{L^{p}(Q_T)}
				\, \left \Vert  \sigma ^{1/p^{\prime}} \chi_{\text {Spt} \,  u}\right \Vert _{L^{{p}^{\prime}}(Q_T)}
			\end{equation}
			holds, where $C=C_0 \vert D \vert^{\frac{1}{n}-\frac{q-1}{p}} $ and $C_0 $ depends on $n, q; $ the $\text{Spt} \, u $ denotes a support of $u.$
		\end{lem}

		\bigskip
		
	Lemma \ref{lem1} establishes the following important result, which is crucial in our proof of the Moser inequality presented in Section \ref{s2}:
		
		\begin{cor} \cite{FM} \label{cc9}
			For a function $u \in Lip_0(D)$ 
			the inequality
			\begin{align}\label{equ2}
				\left  \Vert  u  \right \Vert _{L^{(n+p)/n}(Q_T)} 
				\leq C \left ( \sup_{t\in (t_0-T, t_0)}\,  \left \Vert u(t, \cdot ) \right \Vert _{L^{p}\left (D \right )}^{p/(n+p)} \right )\nonumber\\
    \times  \left \Vert  \omega ^{1/p}  \partial _x u  \right \Vert _{L^{p}(Q_T)}^{n/(n+p)}
				\, \,  \left \Vert  \sigma ^{1/p^{\prime}} \chi_{\text {Spt} \,  u}\right \Vert _{L^{{p}^{\prime}}(Q_T)}^{n/(n+p)}
			\end{align} holds, 
			where $C_0$ depends only on $n. $
		\end{cor}

		\bigskip
		
	The main focus of our paper is to provide a proof of Moser's inequality for positive sub-solutions of equation \eqref{PE} on specific cylindrical domains. This inequality will be instrumental in establishing the Harnack inequality within the cylinders $Q_T^{(t, x)}$, defined as $K_R^x \times (t-T, t)$. It is important to note that the value of $T$ may differ from $R^p$ and varies from point to point, depending on the degeneracy characteristics of the weight function $\omega(t, x).$

		We will utilize the following widely recognized iteration lemma.
	
	\bigskip
		\begin{lem}\label{itlm} Let $f(t)$ be a bounded function in $[\tau_0, \tau_1]$ with $ \tau_0\geq 0. $ Let $\alpha\in [0, 1) $
			and $A>0$ is a positive constant. If for all $ s, t $ such that $ \tau_0\leq s<t \leq \tau_1, $
			$$
			f(s)\leq \alpha f(t)+\frac{A}{(t-s)^\beta}
			$$
			then
			$$
			f(s) \leq c(\alpha, \beta) \frac{A}{(t-s)^\beta}
			$$
			for positive constant $c(\alpha, \beta )>0 .$

		\end{lem}
		
	\begin{proof}
		See, e.g. in \cite{FM}
		\end{proof}
		\bigskip
		
	To establish the Harnack inequality, we will require the following lemma as a crucial component of the proof.
		
		\begin{lem} \label{Bombieri} Let $u: Q_1 \to \R^+ $ be a positive measurable function and the one parameter family of cylinders $\{Q_s\}_{s\in [1/2, 1]}$ with $Q_s \subset Q_r$ and $1/2\leq s<r \leq 1$ be such that
			\begin{itemize}
				\item there exist positive constants $C_1, \theta $ such that for any $0<\delta<1$ it holds the inequality
				\begin{equation}\label{B1}
					\sup_{Q_s} u^\delta \leq \frac{C_1}{(r-s)^ \theta} \cdot \frac{1}{w ( Q_1)} \int_{Q_r} u^\delta  \,  w \, dt dx; 
				\end{equation}
				\item there exists a positive constant $C_2$ such that for any $k>0$ it holds
				\begin{equation}\label{B2}
					w \left ( \left \{   x\in Q_1: \, \, \ln u >k \right \} \right )<\frac{C_2}{k} w ( Q_1) .
				\end{equation}
				Then there exists $C_\theta$ such that 
				\begin{equation}\label{Bom}
					\sup\limits_{Q_s} \, u \leq \exp \left (C_\theta\frac{32 C_2 C_1^4}{(r-s)^{4\theta}} \right ), \quad 1/2\leq s<r \leq 1.
				\end{equation}
				
			\end{itemize}
			
		\end{lem}
		\begin{proof}
		See, e.g. in \cite{FM}
	\end{proof}
\bigskip
	\begin{lem}\cite{FM}\label{Mamedov}
	Let $v$ be a function in Lip($Q_T $),  $D^{+}=Q_{T}\bigcap\{v>0\}\ $ and $D^{+}(s)=Q^{s}_{T}\bigcap D^{+}$. Then the following inequality holds
	$$
	\iint_{D^+(s)} v \, dt dx\leq R \iint\limits_{D^+(s)} \left \vert \nabla_{x} v \right \vert  \, dt dx + T\int\limits_{K_R^{x_0}\cap D^+(s)} v \Big \vert _{t=t_2} \, dx  , 
	$$
 where \quad $Q_T^s=K_{s R}^{x_0}\times \left (t_0-sT, t_0 \right ),\quad \frac{1}{2}< s< 1.$
\end{lem}
\bigskip
Before starting main result, we would like give some well-known properties of the Steklov averages:

$\bullet$\quad Let $v\in L^{1}(Q_T)$ and let $0<h<T$.  The Steklov average $v_{h}(\cdot,t)$ are defined by
\begin{align*}
	v_{h}(t,x)=\frac{1}{h}\int_{t}^{t+h}v(\cdot,\tau)d\tau\quad for\quad 0<t<T-h.
\end{align*}
\begin{lem}\label{stek}
	Let $v\in L^{p}(t_{0}-T,t_{0};W^{1}_{p}(D,\omega dx))$. Then as $h\to0,$ $v_h\to v$ in $ L^{p}(t_{0}-T,t_{0};W^{1}_{p}(D,\omega dx))$. If $v_h\to v$ in norm of space $L^{\infty}(t_{0}-T,t_{0};L^{p}(D))$ with $D=K^{x_{0}}_{R}$, then
	\begin{align*}
		&\iint_{Q_T}\vert\nabla_{x}(v_{h}-v)\vert^{p}\omega(t,x)dtdx\to0,\\
		&\sup_{t\in (t_0-T, t_0)}\Vert v_{h}(t,\cdot)-v(t,\cdot)\Vert_{L^{p}(D)}\Vert\to0, \quad h\to0.
	\end{align*}
	\end{lem}
			\smallskip
		\section{Moser's inequality} \label{s2}
		
		Now we proceed to the Moser inequality, drawing inspiration from the proof techniques employed in \cite{AM1, AM2,FM}.
		
		Assuming that condition \eqref{e4} holds for the function $\omega(t, x)$ and equation \eqref{PE}, we will establish the Moser inequality.
		
		For $(t_0, x_0) \in \R^{n+1}$ and $ R>0 $ define the parabolic cylinder $ Q_T$  (or $  Q_T^{(t_0, x_0)}$ ) be 
		\begin{equation}\label{trek} 
			Q_T=K_R^{x_0}\times ( t_0-T, t_0 ) \quad \text {with } \quad T=C R^p \Big /  \left ( \Xint -_{Q_T^{(t_0, x_0)}} \omega^{\alpha} \, dxdt \right )^{1/{\alpha}} \cdot \end{equation}
		
		For $1/2\leq s<\tau \leq 1$ and the cylinder $ Q_T $ define the congruent cylinders  \,
		 \begin{equation}
		 	 \label{tre} Q_T^s=K_{s R}^{x_0}\times \left (t_0-s T, t_0 \right )  \quad \text {and} \quad Q_T^\tau= K_{\tau R}^{x_0}\times  (t_0 - \tau T, t_0) .
		\end{equation} 
		
		\begin{teo}\label{ta2}
			Let $ u(t, x) $ be a positive sub-solution of \eqref{PE} on $ Q_T. $ Let the conditions \eqref{e4} be fulfilled and the cylinders $\{ Q_T\} $ are of \eqref{trek}, \eqref{tre}.
			Then there exists a constant $C$ depending only on $n, c_1, c_2, c_0, p,\alpha, r $ such that for $0<\delta <1$ the inequality
			\begin{equation}\label{d1}
				\underset{Q_T^s}{\mathrm{ess} \sup }\, u \, \leq \, \frac{1}{(t-s)^{\frac{1}{\delta (L-1)}}} \left ( \frac{C}{\vert Q_T\vert }\iint\limits_{Q_T^\tau} u^\delta \, dt dx \right )^{1/\delta}, \quad 1/2\leq s<t \leq 1
			\end{equation} holds; 
			the constant $C$ does not depend on $\delta.$
		\end{teo}

		\noindent
		\begin{proof} 
			Let $ \frac{1}{2} \leq s <\tau \leq 1 $. For a function $f$, denote $f_+=\max ( f(x), 0 )$. Consider 
			$ \xi  \in \mathrm{Lip} (Q_T) $, a Lipschitz continuous function vanishing on $ Q_T^s $, equaling one on the parabolic boundary $\Gamma ( Q_T^\tau )$, and satisfying $\vert \partial_t  \xi \vert \leq  \frac{C} {(\tau-s)T}$ and $\vert \nabla_{x} \xi \vert \leq \frac{C} {(\tau-s)R}$. Define 
			$$
			v_k=(u-k-M_\tau \xi)_+,  \quad    0 \le k \le \, \sup\limits_{Q_T^\tau} \, \, \left (u-M_\tau \xi  \right )_+,   
			$$
			where $ M_\tau= \underset {Q_T^\tau}{\text{ess}\sup }  \, \, u$, and $ \Omega^{k}=\left \{ (t,x)\in Q_T: \,  v_{k}>0 \right \}$.
			
	By using the equality $u-k=v_k + M_\tau \xi $, we have
	\begin{align}
			-\int_{K_R^{x_0}} (v_k + M_\tau \xi)^{p-1} v_{k} \Big \vert _{t=t_0} \, dx -&\iint_{Q_T}  A(t,x,u-k,\nabla_x u) 
			(\nabla u_x-M_\tau\nabla \xi)  \, dt dx \nonumber \\
			+&\iint_{Q_T}(v_k + M_\tau \xi)^{p-1}\partial _{t}v_{k}\, dt dx  \geq  0 . 
	\end{align}
		By the same way,
		\begin{align}
	-\int_{K_R^{x_0}} (v_k + M_\tau \xi)^{p} \Big \vert _{t=t_0} \, dx
   +&\int_{K_R^{x_0}} (v_k + M_\tau \xi)^{p-1}  M_\tau \xi \Big \vert _{t=t_0} \, dx\nonumber\\
	-&\iint_{Q_T} A(t,x,u,\nabla_x u)\nabla_x u \, dt dx\nonumber\\
 +&	M_\tau \iint_{Q_T} A(t,x,u-k,\nabla_x u) \nabla_{x} \xi  \, dt dx\nonumber\\ 
			+&\iint_{Q_T}(v_k + M_\tau \xi)^{p-1}\partial _{t}(v_{k}+ M_\tau \xi)\, dt dx\nonumber\\
   -&M_\tau\iint_{Q_T}(v_k + M_\tau \xi)^{p-1}\partial _{t}\xi \, dt dx \geq  0 . 
	\end{align} 
	therefore, we have
\begin{align}
		-\int_{K_R^{x_0}} (v_k + M_\tau \xi)^{p} \Big \vert _{t=t_0} \, dx+&\int_{K_R^{x_0}} (v_k + M_\tau \xi)^{p-1}  M_\tau \xi \Big \vert _{t=t_0} \, dx\nonumber\\
		 -&c_{1}\iint_{Q_T} \omega(t, x )\vert\nabla_{x} u\vert^{p} \, dt dx\nonumber\\
   +&c_{2}M_\tau \iint_{Q_T} \omega(t, x )\vert\nabla_{x} u\vert^{p-1} \nabla_{x} \xi  \, dt dx\nonumber\\
   +&\frac{1}{p}\iint_{Q_T}\partial _{t}(v_k + M_\tau \xi)^{p}\, dt dx\nonumber\\
   -&M_\tau\iint_{Q_T}(v_k + M_\tau \xi)^{p-1}\partial _{t}\xi \, dt dx \geq  0 . 
\end{align}
Evaluating fourth integral respect to t, we get
\begin{align}
		\frac{1-p}{p}\int_{K_R^{x_0}} (v_k + M_\tau \xi)^{p} \Big \vert _{t=t_0} \, dx-&\frac{1}{p}\int_{K_R^{x_0}} (v_k + M_\tau \xi)^{p} \Big \vert _{t=t_0 -T} \, dx\nonumber\\
		+&\int_{K_R^{x_0}} (v_k + M_\tau \xi)^{p-1}  M_\tau \xi \Big \vert _{t=t_0} \, dx\nonumber\\
  -&c_{1}\iint_{Q_T} \omega(t, x )\vert\nabla_{x} u\vert^{p} \, dt dx\nonumber\\
		+&c_{2}M_\tau \iint_{Q_T} \omega(t, x )\vert\nabla_{x} u\vert^{p-1}\nabla_{x} \xi  \, dt dx\nonumber\\
  -&M_\tau\iint_{Q_T}(v_k + M_\tau \xi)^{p-1}\partial _{t}\xi \, dt dx \geq  0 . 
\end{align}
Obviously, we have
\begin{align}
		\frac{p-1}{p}\int_{K_R^{x_0}} (v_k + M_\tau \xi)^{p} \Big \vert _{t=t_0} \, dx +&\frac{1}{p}\int_{K_R^{x_0}} (v_k + M_\tau \xi)^{p} \Big \vert _{t=t_0 -T} \, dx\nonumber\\
		 +&c_{1}\iint_{Q_T} \omega(t, x )\vert\nabla_{x} u\vert^{p} \, dt dx\nonumber\\
   \leq&	c_{2}M_\tau \iint_{Q_T} \omega(t, x )\vert\nabla_{x} u\vert^{p-1}\vert \nabla \xi\vert  \, dt dx\nonumber\\
		+&M_\tau\iint_{Q_T}(v_k + M_\tau \xi)^{p-1}\vert\partial _{t}\xi \vert\, dt dx\nonumber\\
  +&\int_{K_R^{x_0}} (v_k + M_\tau \xi)^{p-1}  M_\tau \xi \Big \vert _{t=t_0} \, dx . 
\end{align}
By applying the Young inequality to the last integral, and using the estimate $v_k\leq M_\tau, $ we get
\begin{align}
		c_{1}\iint_{Q_T} \omega(t, x )\vert\nabla u\vert^{p} \, dt dx \leq& \frac{(p-1)c_{2}}{p}\iint_{Q_T} \omega(t, x )\vert\nabla_{x} u\vert^{p} \, dt dx\nonumber \\
		+&\frac{c_{2}}{p}M_\tau^{p} \iint_{Q_T} \vert \nabla_{x} \xi\vert^{p}  \, dt dx\nonumber\\
  +&2^{p-1}M_\tau^{p}\iint_{Q_T} \vert\partial _{t}\xi \vert \, dt dx. 
\end{align}
It follows
\begin{align}
		c_{1}(p)\iint_{Q_T} \omega(t, x )\vert\nabla_{x} u\vert^{p} \, dt dx \leq& c_{2}(p)M_\tau^{p} \iint_{Q_T} \vert \nabla_{x} \xi\vert^{p}  \, dt dx\nonumber\\
		+&2^{p-1}M_\tau^{p}\iint_{Q_T} \vert\partial _{t}\xi \vert \, dt dx. 
\end{align}
	Therefore,  $\vert \nabla_{x}  v_{k}\vert^{p}\leq2^{p-1}(\vert \nabla _{x} u\vert^{p}+M_\tau^{p}\vert \nabla _{x}\xi\vert^{p}) $ implies to
\begin{align}
		c_{1}(p)\iint_{Q_T} \omega(t, x )\vert\nabla_{x} v_k\vert^{p} \, dt dx \leq& c_{2}(p) \iint_{Q_T} \vert \nabla_{x} \xi\vert^{p}  \, dt dx\nonumber\\
		+&2^{p-1}M_\tau^{p}\iint_{Q_T} \vert\partial _{t}\xi \vert \, dt dx. 
\end{align} 
Thus
			\begin{equation}\label{d9}
				\iint_{Q_T} \omega \vert \nabla_{x} v_k \vert ^p dt dx \leq  \frac{ c_{2}M_\tau^{p}}{c_{1}(\tau-s)^p}\left ( \frac{ \omega (\Omega^k)}{R^p}+ \frac{\vert \Omega^k \vert }{ T}\right), 
			\end{equation}
			$$
			1/2\leq s<\tau\leq 1, \quad 0\leq k\leq \sup\limits_{Q_T} \left ( u-M_\tau \xi \right )_+ \,   .
			$$
			Based on Corollary \ref{cc9}, this inequality yields 
			$$
			\Vert v_k \Vert_{L^{1}(Q_T)} \leq \Vert v_k \Vert_{L^{(n+p)/n}(Q_T)}\,  \vert \Omega^k \vert ^{p/(n+p)}\, 
			$$
			$$
			\leq C \left ( \sup _{t\in (t_0-T, t_0)}\Vert  v_k(t, \cdot )\Vert _{L^{p}(D)}^{p/(n+p)} \right ) \times $$ \medskip 
			$$ \times \Vert \omega^{1/p} \nabla_{x} v_k  \Vert _{L^{p}(Q_T)}^{n/(n+p)} \, \, \Vert \sigma^{1/p^{\prime}} \chi _{\text{Spt}\, v_k }  \Vert _{L^{{p}^{\prime}}(Q_T)}^{n/(n+p)} \vert \Omega^k \vert^{p/(n+p)}  
			$$
			\medskip
			$$
			\leq C \frac{M_\tau}{\tau-s}\bigg(\frac{c_2}{c_1}\bigg)^{\frac{1}{p}}
			\left ( \frac{\omega(\Omega^k)}{R^p}+\frac{\vert \Omega^k \vert }{T}\right )^{1/p}\vert \Omega^k\vert^{p/(n+p)} \sigma(\Omega^k)^{n/p^{\prime}(n+p)} .
			$$
			
			\medskip
			
			Set $$y(k)=\iint_{Q_T} v_k dtdx=\iint_{\Omega^k} v_k dtdx  \quad \text{ then}  \quad y^\prime (k)=-\vert \Omega^k \vert. $$
			
			Then
			\begin{equation}\label{equi2}
				y(k) \leq   \frac{CM_\tau}{\tau-s}\left ( \frac{\omega(\Omega^k)}{R^p}+\frac{\vert \Omega^k \vert }{T}\right )^{1/p}\vert \Omega^k\vert^{p/(n+p)} \sigma(\Omega^k)^{n/p^{\prime}(n+p)} 
			\end{equation}
			Consider the numbers $\alpha>\frac{n+p}{p}$ and $r>\frac{n(p-1)}{p}$ satisfying the condition:
			$$\frac{n(p-1)}{pr}+\frac{n+p}{p\alpha}<1.$$
			Based on Holder's inequality
			$$
			\omega (\Omega^k ) \leq \Vert  \omega \Vert _{L^{\alpha}(Q_T)} \vert \Omega^k \vert^{1/\alpha^\prime} \quad \text{and }\quad 
			\sigma (\Omega^k ) \leq \Vert  \sigma \Vert _{L^r(Q_T)} \vert \Omega^k \vert^{1/r^\prime}, 
			$$
		Hence, from \eqref{equi2}, it follows that
			$$
			y(k) \leq   \frac{CM_\tau}{\tau-s}\left ( \frac{\Vert  \omega \Vert _{L^{\alpha}(Q_T)} \vert \Omega^k \vert^{1/{\alpha}^\prime}}{R^p}+\frac{\vert \Omega^k \vert }{T}\right )^{1/p}\vert \Omega^k\vert^{p/(n+p)}\times $$$$ \times \left ( \Vert  \sigma \Vert _{L^{r}(Q_T)} \vert \Omega^k \vert^{1/r^\prime} \right )^{n/p^{\prime}(n+p)}.
			$$
			Thus 
			\begin{equation}\label{e3.22}
				y(k) \leq   \frac{CM_\tau}{\tau-s}\left ( \frac{\Vert  \omega \Vert _{L^{\alpha}(Q_T)} }{R^p}+\frac{\vert Q_T \vert^{\frac{1}{\alpha}} }{T}\right )^{1/p}\vert \Omega^k\vert^{\frac{p}{n+p}+\frac{1}{p{\alpha}^\prime}+\frac{n}{p^{\prime}(n+p)r^\prime}} \Vert  \sigma \Vert _{L^r(Q_T)} ^{\frac{n}{p^{\prime}(n+p)}}.
			\end{equation}
		Select the height $T$ of the cylinder $Q_T$ based on the following inequality
			$$
			\frac{\Vert  \omega \Vert _{L^{\alpha}(Q_T)} }{R^p}= \frac{C \vert Q_T \vert^{\frac{1}{\alpha}} }{T}, \quad \text {i.e. } \quad T= C R^p\Big/ \left (\Xint -_{Q_T}\omega^\alpha \, dxdt \right )^{1/{\alpha}} .
			$$
			That, together with \eqref{e3.22}, leads us to the inequality
			$$
			y(k) \leq   \frac{CM_\tau}{\tau-s}\left ( \frac{\Vert  \omega \Vert _{L^{\alpha}(Q_T)}\Vert  \sigma \Vert _{L^r(Q_T)} ^{\frac{n(p-1)}{(n+p)}} }{R^p}\right )^{1/p}\vert \Omega^k\vert^{\frac{p}{n+p}+\frac{1}{p{\alpha}^\prime}+\frac{n}{p^{\prime}(n+p)r^\prime}} .
			$$
			By applying the Muckenhoupt condition \eqref{e4} throughout the cylinders \eqref{cyl}, we obtain
			$$
			y(k) \leq   \frac{CM_\tau}{\tau-s}\left ( \frac{\Vert  \omega \Vert _{L^{\alpha}(Q_T)}\Vert  \sigma \Vert _{L^r(Q_T)} ^{\frac{n(p-1)}{(n+p)}} }{R^p}\right )^{1/p}\left (-y^\prime(k) \right )^{\frac{p}{n+p}+\frac{1}{p{\alpha}^\prime}+\frac{n}{p^{\prime}(n+p)r^\prime}} .
			$$
			Let $L=\frac{p}{n+p}+\frac{1}{p{\alpha}^\prime}+\frac{n}{p^{\prime}(n+p)r^\prime} $ then 
			$$
			L-1=\frac{1}{n+p} \left ( 1-\frac{n+p}{p{\alpha}}-\frac{n(p-1)}{pr}\right ),
			$$
			by assumptions on $p, r$ we have $L>1.$ 
			
			Denoting $B=\frac{1}{R^p}\Vert  \omega \Vert _{L^{{\alpha}}(Q_T)}\Vert  \sigma \Vert _{L^r(Q_T)} ^{\frac{n(p-1)}{(n+p)}} $
			we have
			$$
			y(k) \leq   \frac{CB^{1/p} M_\tau}{\tau-s}\left (-y^\prime(k) \right )^L,
			$$
			or 
			$$
			1\leq \left (\frac{CB^{1/p} M_\tau}{\tau-s} \right )^{1/L} \frac{-y^\prime(k)}{y(k)^{1/L}}.
			$$
			Integrating this inequality over the interval $(0, \sup\limits_{\tau Q_T} v_0 )$, 
			we get
			$$
			M_s \leq \sup\limits_{\tau Q_T} v_0 \leq   
			y(0)^{\frac{L-1}{L}} \left (\frac{CB^{1/p} M_\tau}{\tau-s} \right )^{1/L} \frac{L}{L-1}, \quad 1/2\leq s<\tau \leq 1 \cdot $$
			From this, utilizing the Young inequality, we conclude that
			$$
			M_s\leq  \frac{1}{L} M_\tau+y(0) \, \frac{ \left (CB^{1/p} \right )^{\frac{1}{L-1}}}{(\tau-s)^{\frac{1}{L-1}}}, \left ( \frac{L}{L-1}\right )^{1/(L-1)}, \quad 1/2\leq s<\tau \leq 1 \cdot 
			$$
			Based on the iteration Lemma \ref{itlm}, this suggests
			$$
			M_s\leq   \, \frac{ \left (C_2B^{1/p} \right )^{\frac{1}{L-1}}}{(\tau-s)^{\frac{1}{L-1}}}\, y(0), \quad 1/2\leq s<\tau \leq 1 \cdot 
			$$
		Now, by utilizing the condition \eqref{e4} to estimate $\vert Q_T \vert=c_nT R^n$, and considering the choice of $T$ from \eqref{Tdef}, it can be demonstrated that
			$$
			B^{1/p(L-1)}=\vert Q_T \vert^{-1}.
			$$ Indeed, 
			$$
			B=\frac{c_0^{n/(n+p)}}{R^p}\left ( \frac{1}{|Q_T|}\iint_Q \omega^{{\alpha}} \, dt dx \right )^{\frac{p}{{\alpha}(n+p)}} \vert Q_T\vert^{\frac{1}{{\alpha}}+\frac{n(p-1)}{r(n+p)}} 
			$$
			$$
			=|Q_T|^{\frac{p}{n+p}\left (-1+\frac{n+p}{p{\alpha}}+\frac{n(p-1)}{pr} \right )}, 
			$$
			then
			$$
			B^{\frac{1}{p(L-1)}}=|Q_T|^{-1}.
			$$
			
			Therefore, 
			$$
			M_s\leq   \, \frac{1}{(\tau-s)^{\frac{1}{L-1}}} \,  \frac{C_3}{\vert Q_T \vert }\, y(0), \quad 1/2\leq s<\tau \leq 1 ,
			$$
			and using the estimate $ y(0)\leq  \iint_{Q_T^{\tau}} \, v \, dt dx$ it follows 
			$$
			M_s\leq   \, \frac{C_3}{(\tau-s)^{\frac{1}{L-1}}} \, \frac{1}{\vert Q_T \vert } \iint_{Q_T^{\tau}} \, v \, dt dx, \quad 1/2\leq s<\tau \leq 1,
			$$

			Using further the inequality  
			$$
			\iint_{\Omega^0} v_0 \, dt dx = \iint_{Q_T^\tau} ( u-M_\tau \xi ) \, dt dx  \leq M_\tau^{1-\delta} \iint\limits_{Q_T^\tau} u^{\delta} \, dt dx, \quad 0<\delta<1, 
			$$
			and the Young inequality, we conclude
			$$
			M_s \leq  (1-\delta) M_\tau +  \delta \left  ( \frac {C_3}{(\tau-s)^{\frac{1}{L-1}}} \,  \, 
			\frac{1}{\vert Q_T\vert }\iint\limits_{Q_T^\tau}  u^{\delta} \, dt dx \right )^{\frac{1}{\delta}},
			$$
			which also on basis of the iteration Lemma \ref{itlm} implies
			$$
			M_{s}\leq \frac{1}{\left (\tau-s \right )^{\frac{1}{\delta (L-1)}}} \left ( \frac{C_1}{\vert Q_T\vert } \iint\limits_{Q_T^\tau} u^\delta dt dx \right)^{1/ \delta}, \quad 1/2\leq s< \tau \leq 1, $$
			with a constant $C_1$ depending only on $n, c_0, c_1, c_2, p,\alpha, r . $
			
			This completes the proof of Theorem \ref{ta2}.
		\end{proof}

		\section{Harnack's inequality}
		
		\begin{teo}\label{mth} 
			Let $u(t, x)$ be a positive weak solution of the equation \eqref{PE}. Let the conditions  \eqref{e4} be fulfilled all over the cylinders \eqref{cyl} and equation \eqref{PE}.  Then there exists a positive constant $C$ depending only $n, c_0, c_1, c_2, p, r $ such that the inequality
			\begin{equation}\label{Harn}
				\frac{\underset{Q_{T, \frac{1}{4}}} {\mbox{ess}\, \sup } \, \, u }{\underset{Q_T^{\frac{1}{4}}} {\mbox{ess}\, \inf } \, \, u } \, \leq C
			\end{equation}
			holds. Here the small lower cylinder $Q_{T, \frac{1}{4}}=K_{R/4}^{x_0} \times \left  ( t_0-3T/4, \, t_0-T/2 \right ), $ and  upper cylinders $ Q_T^{\frac{1}{4}}=K_{R/4}^{x_0}\times (t_0-T/4, t_0) .$
		\end{teo}
		\begin {proof} We aim to establish Harnack's inequality for \eqref{PE} based on Lemma \ref{Bombieri}. Let $ l>0 $ be such that the surface $u(y, x)=l$ divides $Q_T^{\frac{1}{4}}$ into $n+1$ equal parts in terms of Lebesgue measure, meaning:
		\begin{equation}\label{e5.1}
			l=\sup \left \{ k\in \R: \quad \left \vert \left \{ u< k \right \} \cap Q_{T, \frac{1}{4}}  \right \vert \leq \frac{1}{2}\left \vert Q_{T, \frac{1}{4}} \right  \vert \right \} \cdot  
		\end{equation}
		By scaling the solution $u(t, x)$ by a constant, we may assume that $l=1$.
		
		Let $ v=\ln^+ \frac{1}{u} $ and $ D^+=Q_T \cap \{ v>0 \}$. 
		Define $1/2\le s<\tau \le 1$. Let $ 0 \leq \eta \leq 1$ be a Lipschitz continuous function that is unity on $Q_T^s=K_{sR}^{x_0}\times (t_0-sT, t_0)$ and zero on $\Gamma (Q_T^\tau)$. We assume $\eta$ satisfies $\vert \nabla_{x} \eta \vert \leq \frac {c_n}{(\tau-s)R}$ and $\vert \partial_t \eta \vert \leq \frac {c_n}{(\tau-s)T}$. It is evident that such a function $\eta$ exists.

	Choose a test function $\left (\frac{1}{u^{p-1}} -1 \right )_+\eta^p $ in \eqref{dfn}:
		\begin{equation}\label{fei}
			\begin{split}
				-&\int_{K_{\tau R}^{x_0}} u^{p-1} \cdot \left (\frac{1}{u^{p-1}}-1 \right)_+ \eta^p \,  \Big \vert _{t=t_1}^{t=t_2} \, dx\\
				-&\iint_{Q_T^\tau} \eta^p A(t,x,u,\nabla_x u) \nabla_x  \left ( \frac{1}{u^{p-1}}-1\right )_+  \, dt dx \\
				+&\iint_{Q_T^\tau}\eta^p  u^{p-1} \cdot \, \partial _t \left (\frac{1}{u^{p-1}}-1 \right )_+  \, dt dx\\
				+&\iint_{Q_T^\tau} u^{p-1} \cdot p\eta^{p-1}  \left (\frac{1}{u^{p-1}}-1 \right )_+ \,  \partial_t \eta \\
				=&\iint_{Q_T^\tau}   \, p\eta^{p-1} \, A(t,x,u,\nabla_x u)  \left ( \frac{1}{u^{p-1}}-1\right )_+ \,  \nabla_x  \eta  \, dt dx;  \quad \quad  t_1=t_0-T, \, t_2=t_0 \cdot  
			\end{split}
		\end{equation}
	By using the integration by part, we get
		\begin{equation*}
		\begin{split}
		&\int_{K_R^{x_0}} \eta^p \left ((p-1) \ln^+ \frac{1}{u}-(1-u^{p-1})_+ \right )\Big \vert _{t=t_2}\, dx\\
		  +&(p-1)\iint_{Q_T^\tau} \eta^p  u^{-p} A(t,x,u,\nabla_x u) \, \nabla_x u   \, \cdot \chi _{0<u<1}   \, dt dx\\ 
		-&\iint_{Q_T} p \eta^{p-1} \left ((p-1) \ln^+ \frac{1}{u}-(1-u^{p-1})_+ \right ) \partial_t \eta   \,dt dx \\
		=&\iint_{Q_T}   \, p\eta^{p-1} \, A(t,x,u,\nabla_x u)  \left ( 1-u^{p-1}\right )_+ \, \frac{ \nabla_x \eta}{u^{p-1}}  \, dt dx.
			\end{split}
	\end{equation*}
		\bigskip
		
		Therefore, 
		$$
		(p-1)\iint_{Q_T^\tau\cap D^+} A(t,x,u,\nabla_x u)\frac{\nabla_x u}{u^{p}}\eta^p \, dt dx
		$$
		$$
		 + \int_{K_R^{x_0}} \eta^p \left ( (p-1)\ln^+ \frac{1}{u}-(1-u^{p-1})_+ \right )\Big \vert _{t=t_2}\, dx
		$$
		$$
		\leq p \iint_{Q_T^{\tau}}\left ( 1-u^{p-1} \right )_+ \,  \eta^{p-1} A(t,x,u,\nabla_x u)\frac{\nabla_x \eta}{u^{p-1}}\, dt dx 
		$$
		$$
		+\iint_{ Q_T^\tau}   p\left [(p-1) \ln^+ \frac{1}{u}-(1-u^{p-1})_+ \right ]\eta^{p-1} \vert \partial_t \eta \vert \, dt dx .
		$$
		Base on the growth  conditions \eqref{grw1} and \eqref{grw2}, it follows
		$$
		(p-1)c_{1}\iint_{Q_T^\tau\cap D^+} \omega(t,x) \left\vert\frac{\nabla_x u}{u}\right\vert^{p}\eta^p \, dt dx
		$$
		$$
		+ \int_{K_R^{x_0}} \eta^p \left ( (p-1)\ln^+ \frac{1}{u}-(1-u^{p-1})_+ \right )\Big \vert _{t=t_2}\, dx
		$$
		$$
		\leq c_{2}p \iint_{Q_T^{\tau}}\omega(t,x)\left ( 1-u^{p-1} \right )_+ \,  \eta^{p-1} \left\vert\frac{\nabla_x u}{u}\right\vert^{p-1}\nabla_x \eta\, dt dx 
		$$
		$$
		+\iint_{ Q_T^\tau}   p\left [(p-1) \ln^+ \frac{1}{u}-(1-u^{p-1})_+ \right ]\eta^{p-1} \vert \partial_t \eta \vert \, dt dx .
		$$
		Thus 
		$$
		c_1(\varepsilon) \iint_{Q_T^\tau} \left \vert \nabla_{x}  \ln^+ \frac{1}{u} \right \vert ^p \omega \eta^p \, dt dx 
		+\int_{K_{\tau R}^{x_0}}\left ( (p-1)\ln^+ \frac{1}{u}-(1-u^{p-1})_+ \right )\Big \vert _{t=t_2}\, dx
		$$
		$$
		\leq c_2(\varepsilon) \iint _{Q_T^\tau}\omega \left (1-u^{p-1} \right )_+^p  \vert \nabla_{x} \eta \vert ^p \, dt dx
		$$
		$$
		+\iint_{ Q_T^\tau}   p\left [(p-1) \ln^+ \frac{1}{u}-(1-u^{p-1})_+ \right ]\eta^{p-1} \vert \partial_t \eta \vert \, dt dx ,  
		$$
		therefore and since $ (p-1) \ln^+ \frac{1}{u}-(1-u^{p-1})_+ \geq 0 $, 
		\begin{align}\label{e5.2}
				 \iint_{Q_T^\tau}  \eta^p  \omega \left \vert \nabla_{x}\ln^+\frac{1}{u}  \right \vert^p  dt dx +&\int_{K_{\tau R}^{x_0}}\left ( (p-1)\ln^+ \frac{1}{u}-(1-u^{p-1})_+ \right )\Big \vert _{t=t_2}\, dx \nonumber\\
				&\leq  \frac{c_2(\varepsilon)}{c_1(\varepsilon)} \cdot \frac{c_n}{( \tau-s)^p} \frac{\omega(Q_T)}{R^p}\nonumber\\
    +& \frac{p(p-1)}{c_1(\varepsilon) T (\tau-s)} \iint_{Q_T^\tau} \eta^{p-1} \ln^+ \frac{1}{u}  \, dt dx \, \cdot 
		\end{align}

		Thus for $ \quad 1/2\leq s<\tau\leq 1 $ we have the inequality
		\begin{align}\label{4.3}
			\iint_{Q_T^s} \omega \left \vert\nabla_{x} v \right \vert^p  dt dx
			\leq & \frac{c_2(\varepsilon)}{c_1(\varepsilon)} \cdot \frac{c}{( \tau-s)^p} \frac{\omega(Q_T)}{R^p}\nonumber\\
   +& \frac{p(p-1)}{c_1(\varepsilon) T (\tau-s)} \iint_{Q_T^\tau} v dt dx \, \cdot 
		\end{align}

			Lemma \ref{Mamedov}, combined with Holder's inequality, leads to
			\begin{align}\label{F.M}
				\int_{D^+(s)} v \, dt dx \leq R \iint_{D^+(s)} \left \vert \nabla_{x} v \right \vert \, dt dx\nonumber\\
    +T\int\limits_{K_R^{x_0}\cap D^+(s)} v \Big \vert _{t=t_2} \, dx 
    \le T\int\limits_{K_R^{x_0}\cap D^+(s)} v \Big \vert _{t=t_2} \, dx\nonumber\\
    +\sigma^{\prime}(Q_T)^{1/p^{\prime}} \, R \, \left ( \iint_{D^+(s)}\omega \left \vert \nabla_{x} v \right \vert^p \, dt dx  \right)^{1/p} \cdot
			\end{align}
			Utilizing \eqref{lgr}, \eqref{e5.2}, and \eqref{4.3}, along with the expression of $v=\ln^+ \frac{1}{u}$ and the Young inequality, we deduce that 
			$$
			\iint_{D^+(s)} v  \, dt dx \le \frac{c \, R\,  \sigma(Q_T)^{1/p^{\prime}}  }{(\tau-s)} \cdot \left ( \left ( \frac{c_2}{c_1} \frac{\omega(Q_T)}{R^p} \right )^{1/p}+ \frac{p(p-1)}{c_1 T^{1/p} } \left ( \iint_{D^+(\tau)} v \, dt dx \right )^{1/p}\right ).
			$$
			Utilizing the assumption \eqref {e4} and noting that it implies $\omega \in A_p$ throughout the cylinders \eqref{cyl}, we can employ Young's inequality. This allows us to derive:
			$$
			\iint_{D^+(s)} v  \, dt dx \le \frac{1}{p} \iint_{D^+(\tau)} v\, dt dx +\frac{c \vert Q_T \vert }{(\tau-s)^{p^{\prime}}},
			$$
			which consequently implies 
			\begin{equation}\label{lgr}
				\iint_{D^+(s)} v  \, dt dx \le  \frac{c_3 \vert Q_T \vert }{(\tau-s)^{p^{\prime}}}, \quad 1/2\le s<\tau \le1   
			\end{equation}
			based on Lemma \ref{itlm}. 
			
			Here, we utilized the assumption \eqref{e4} which implies
   $$\omega(Q_T) \sigma^{p-1}(Q_T)\le c \left \vert Q_T \right \vert^p,$$
   
which is equivalent to the condition $\frac{R^{p^{\prime}} \sigma(Q_T)}{T^{\frac{p^{\prime}}{p}}}\le c \vert Q_T \vert$ based on the definition of $T$ from \eqref{Tdef}. Applying \eqref{lgr} for $s=3/4$ and $\tau=1$, we deduce the existence of a constant $C$ depending solely on $c_1, c_2, c_0$, and $n$, such that for any $k>0$, 
			\begin{equation}\label{e5.3}
				\left \vert \left \{  (t, x)\in Q_{T}^{\frac{1}{4}}: \,  \ln \frac{1}{u}  > k  \right \}\right \vert \leq \frac{C}{k} \, \vert Q_T\vert \cdot 
			\end{equation}
			
			When setting $C=1$ in the definition of $T$ in \eqref{Tdef}, and considering that $1/u$ serves as a sub-solution, we can simultaneously apply Moser's estimate \eqref{d1} and the logarithmic estimate \eqref{e5.3}. Consequently, Lemma \ref{Bombieri} asserts the existence of a constant $C$ that relies solely on $c_0, c_1, c_2$, and $n$, ensuring:
			\begin{equation}\label{d5.4}
				\frac{1}{\underset{Q_T^{\frac{1}{4}}} {\mbox{ess}\, \inf } \, u} \, \,  \leq C
			\end{equation}

			Show that, 
			$$
			\underset {Q_{T, \frac{1}{4}} }{\mbox{ess}\, \sup } \, \,  u \leq C,
			$$
			where  
			\begin{equation}\label{ff2}
				Q_{T, \frac{1}{4}}=K_{R/4}^{x_0} \times \left  ( t_0-3T/4, \, t_0-T/2 \right ).
			\end{equation}
		Suppose the maximum value of $u$ in $Q_{T, \frac{1}{4}}$ is achieved at $(t^\prime, x^\prime)$, such that:
		\begin{equation}\label{tpm}
			\sup\limits_{Q_{T, \frac{1}{4}}}\, u=u(t^\prime, x^\prime) \cdot
		\end{equation}  
		Define $T_1$ as follows:
		\begin{equation}\label{tbir}
			\left (\iint\limits _{Q_{T_1}} \omega^{\alpha} \, dtdx\right )^{1/\alpha} =  \frac{C_1 R^{n/{\alpha}+p}}{T_1^{1/{\alpha}^\prime}}, 
		\end{equation}
		where $Q_{T_1}=K_R^{x_0}\times (t^\prime-T_1, t^\prime)$, representing the cylinder \eqref{Tdef} with the upper point $(t^\prime, x_0)$. The value of $C_1$ will be determined later. 
			
			In this context, if $T_1<T/4$, we are within the framework of Moser's estimate \eqref{d1} for the cylinders $\tilde Q_{T_1}^s=K_{sR}^{x_0}\times (t^\prime-sT_1, t^\prime) $ and $\tilde Q_{T_1}^\tau= K_{\tau R}^{x_0}\times (t^\prime-\tau T_1, t^\prime) $, where $1/4\leq s<\tau \leq 1/2$: 
			
			\begin{equation}
				\underset{\tilde Q_{T_1}^s}{\mathrm{ess} \sup }\, u \, \leq \, \frac{C}{(\tau-s)^{\frac{1}{\delta (L-1)}}} \left ( \frac{1}{\vert Q_{T_1}\vert }\iint\limits_{\tilde Q_{T_1}^\tau} u^\delta \, dt dx \right )^{1/\delta}, \quad 1/8\leq s<\tau \leq 1/4 \cdot 
			\end{equation}
		Define $ w=\ln^+ u $. 
		Let $D^-=Q_{T_1}\cap\{ w>0 \}$. 
		Consider a Lipschitz continuous function $\eta$ such that $0 \leq \eta \leq 1$. In this case, $\eta$ is unity on the region $$\tilde{Q}_{T_1}^s=K_{s R}^{x_0}\times \left (t^\prime-sT_1, \, t^\prime \right )$$ and becomes zero outside $$K_{\tau R}^{x_0}\times \left ( t^\prime-\tau T_1, \, t_0-T/4 \right ) . $$
			Suppose $\eta$ satisfies $\vert \nabla_{x} \eta \vert \leq \frac {c}{(\tau-s)R}$ and $\vert \partial_t \eta \vert \leq \frac {c}{(\tau-s)T_1}$, where $1/4\leq s<\tau \leq 1/2$.
			
			Choose a test function $ \frac{1}{u^{p-1}}\eta^p $ in \eqref{dfn}, and let $ t_1=t^\prime-\tau T_1 $ and $ t_2=t_0-T/4 $:
			\begin{align}
			-\int_{K_{\tau R}^{x_0}} u^{p-1} \cdot \frac{1}{u^{p-1}} \eta^p \,  \Big \vert _{t=t_1}^{t=t_2} dx +&(p-1)\iint_{\tilde Q_{T_1}^\tau}A(t,x,u,\nabla_{x}u)\frac{\nabla_{x}u}{u^{p}} \eta^p \, dt dx\nonumber\\
			-&(p-1)\iint_{\tilde Q_{T_1}^\tau}u^{p-1}\frac{\partial_t u}{u^{p}}\eta^{p}\, dt dx\nonumber\\
   +& p\iint_{\tilde Q_{T_1}^\tau} u^{p-1}\cdot \frac{1}{u^{p-1}} \, \eta^{p-1}  \partial_t \eta \, dtdx\\
			 -&p\iint_{ \tilde Q_{T_1}^\tau}  \, A(t,x,u,\nabla_{x}u)\frac{1}{u^{p-1}} \,   \eta^{p-1} \nabla_{x} \eta  \, dt dx =0.
			\end{align}
			Then, we have
			$$
		(p-1)c_{1}\iint_{\tilde Q_{T_1}^\tau}\omega(t,x)\left\vert\frac{\nabla_x u}{u}\right\vert^{p} \eta^p \, dt dx +\iint_{\tilde Q_{T_1}^\tau} p ( \ln u +1) \, \eta^{p-1} \partial_t \eta\, dt dx 
			$$
			$$
			\leq  pc_{2}\iint_{\tilde Q_{T_1}^\tau }  \omega(t,x)\left\vert\frac{\nabla_x u}{u}\right\vert^{p-1}\eta^{p-1}\nabla_x \eta\, dt dx . 
			$$
		This, in conjunction with Young's inequality, leads to
			$$
			\iint_{\tilde Q_{T_1}^\tau } \omega \vert \nabla_{x} \ln u\vert^p  \eta^p dt dx  \leq
			\frac{ c_2}{c_1}\iint_{ \tilde Q_{T_1}^\tau }   \, \omega \vert  \nabla_{x} \eta \vert ^p \, dt dx  
			$$
			$$
			+\frac{p}{c_1} \iint_{\tilde Q_{T_1}^\tau} p \eta^{p-1} \left ( \vert  \ln u \vert +1 \right ) \,  \vert \eta_t \vert  dt dx \cdot
			$$
			Similar to the previous analysis, for $ \frac{1}{4}\leq s<\tau \leq \frac{1}{2}$ and the function $v=\vert \ln u \vert $, we can derive the estimate
			\begin{align}\label{ee4.9}
				\iint_{\tilde Q_{T_1}^s} \omega \left \vert \nabla_{x} v \right \vert ^p  \, dt dx \leq &\frac{pc_n}{T_1(\tau-s)c_1} \iint_{\tilde Q_{T_1}^\tau} v  \, dt dx\nonumber\\ 
				+&\frac{c_2}{c_1} \frac{\omega(Q_{T_1}) }{ R^p \, (\tau-s)^p}+\frac{pc_{n}}{T_{1}(\tau-s)c_{1}}\vert Q_{T_1}\vert.
			\end{align}
		On the other hand, employing \eqref{e4}, we deduce that $$\frac{\vert Q_{T_1}\vert}{T_1} \asymp R^n = \frac{R^{n+p}}{R^p}\asymp\frac{\omega (Q_{T_1})}{R^p}, $$ hence, combined with \eqref{ee4.9}, we obtain
			\begin{equation}\label{eee4.9}
				\iint_{\tilde Q_{T_1}^s} \omega \left \vert \nabla_{x} w \right \vert ^p  \, dt dx \leq \frac{pc_n}{T_1(\tau-s)c_1} \iint_{\tilde Q_{T_1}^\tau} v  \, dt dx 
				+ C \frac{\omega(Q_{T_1}) }{ R^p \, (\tau-s)^p} .
			\end{equation}
		Using similar reasoning as in \eqref{F.M} and \eqref{e5.1}, we derive the following inequality 
			\begin{align}\label{e4.8}
				\iint\limits _{ \tilde Q_{T_1}^s} v \, dt dx \leq& R \iint\limits _{ \tilde Q_{T_1}^s} \left \vert \nabla_{x} v \right \vert \, dt dx\nonumber\\
    \le& \sigma(Q_{T_1})^{1/p^{\prime}} \, R \, \left ( \iint\limits _{\tilde Q_{T_1}^s}\omega \left \vert \nabla_{x} v \right \vert^p \, dt dx .  \right)^{1/p} .
			\end{align}
			Thus, utilizing \eqref{ee4.9}, \eqref{e4}, and the Young inequality, we obtain
			$$
			\iint\limits_{\tilde Q_{T_1}^s} v  \, dt dx \le \frac{c \, R \sigma(Q_{T_1})^{1/p^{\prime}}  }{(\tau-s)} \left (  \left (  \frac{c_2}{c_1} \frac{\omega(Q_{T_1})}{R^p} \right )^{1/p}+ \frac{p}{c_1 T_1^{1/p} } \left ( \iint\limits_{\tilde Q_{T_1}^\tau} v \, dt dx \right )^{1/p}\right ),
			$$
			Combining this observation with Lemma \ref{itlm}, we can infer
			$$
			\iint\limits _{\tilde Q_{T_1}^s} v  \, dt dx \le \frac{1}{p} \iint\limits _{\tilde Q_{T_1}^\tau } v \, dt dx +\frac{c \vert Q_{T_1} \vert }{(\tau-s)^{p^{\prime}}} \, ,
			$$
			or
			$$
			\iint\limits_{ \tilde Q_{T_1}^s} v  \, dt dx \le \frac{C}{(\tau-s)^{p^{\prime}}} \vert Q_{T_1} \vert  \cdot
			$$   
			This inequality straightforwardly implies the existence of a positive constant $C$ determined solely by $c_1, c_2, c_0$, and $n$. For any $k>0$, it is satisfied that
			\begin{equation}\label{ee5.3}
				\left \vert \left \{   (t, x)\in  \tilde Q_{T_1}^{\frac{1}{4} }:  \, \ln u > k \right \} \right \vert  <  \frac{C}{k} \, \vert Q_{T_1} \vert \cdot 
			\end{equation}
		Using Moser's inequality \eqref{d1} specifically for the cylinders, 
			$$
			\tilde{Q}_{T_1}^s=K_{sR}^{x_0}\times \left ( t^\prime-s T_1, \,  t^\prime \right ) $$ 
			$$
			\tilde{Q}_{T_1}^\tau=K_{\tau R}^{x_0}\times \left ( t^\prime-\tau T_1, \,  t^\prime \right ) 
			$$
			we can deduce
			\begin{equation}\label{di}
				\underset{\tilde Q_{T_1}^s}{\mathrm{ess} \sup }\, u \, \leq \, \frac{C}{(\tau-s)^{\frac{1}{\delta (L-1)}}} \left ( \frac{1}{\vert Q_{T_1}\vert }\iint\limits_{\tilde Q_{T_1}^\tau} u^\delta \, dt dx \right )^{1/\delta}, \quad 1/8\leq s<\tau \leq 1/4 \cdot 
			\end{equation}

		Following the establishment of estimates \eqref{ee5.3} and \eqref{di}, we deduce below based on Lemma \ref{Bombieri}: there exists a constant $C$ that relies solely on $c_0$, $c_1$, $c_2$, and $n$, such that
		 \[\underset{\tilde {Q}_{T_1}^{\frac{1}{4}}} {\mbox{ess}\, \sup } \, \,  u \leq C.\]
		Let's remember that $ \sup\limits_{Q_{T, \frac{1}{4}}} \, u =u(t^\prime, x^\prime)\leq \sup\limits_{\tilde Q_{T_1}^{\frac{1}{4}}} \, u  $, where $Q_{T, \frac{1}{4}}$ is defined as in \eqref{ff2}. Therefore,
			\begin{equation}\label{f1}
				\sup\limits_{Q_{T, \frac{1}{4}}} \, u \leq C \,  \cdot
			\end{equation}
		Deriving from \eqref{d5.4} and \eqref{f1}, we obtain
			\begin{equation}\label{e5.4}
				\frac{\underset{Q_{T, \frac{1}{4}}} {\mbox{ess}\, \sup } \, \, u }{\underset{Q_T^{\frac{1}{4}}} {\mbox{ess}\, \inf } \, \, u } \, \leq C
			\end{equation}
			
		The proof of Theorem \ref{mth} is finalized once we establish $T_1<T/4$, as this ensures that the maximum point $(t^\prime, x^\prime)$ falls within the cylinder $\tilde Q_{T_1}^{\frac{1}{4}}$.
		
		In the case where $T_1\geq T/4$, we leverage the definition of $T_1$ from equation \eqref{tbir}, along with the doubling property of the functions satisfying \eqref{e4}. This can be accomplished using conventional methods.
			$$
			\iint\limits_{Q_{T_1}} \omega^p \, dtdx  \geq C_2 \left ( \frac{\vert Q_{T_1}\vert }{\vert Q_T\vert }\right )^{\delta_1} \iint\limits_{Q_{T}} \omega^p \, dtdx ,
			$$
			we get  
			$$
			C_2 \left ( \frac{T_1}{T}\right )^{p-1+\delta_1}  \leq C_1, 
			$$
			therefore,  
			\begin{equation}\label{dfT1}
				\left (1/4\right )^{p-1+\delta_1} \leq  C_1 / C_2 \cdot
			\end{equation}
			
		By selecting an adequately small constant $C_1$ from \eqref{tbir}, we arrive at the conclusion that $C_1$ is less than $C_2 \left ( 1/4 \right )^{1/p^\prime+\delta_1/p}$. This discrepancy contradicts the assumption. Consequently, it demonstrates that $T_1$ is not greater than $T/4$.
			
			This completes proof of the Harnack inequality.
		\end{proof}

	\end{document}